\documentclass[12pt]{amsart}
\usepackage{amsmath,amssymb,amsthm,amsfonts,mathabx,enumerate,array,color,lscape,fancyhdr,layout,pst-all}
\usepackage{xcolor}
\usepackage[english]{babel}
\usepackage{young}
\usepackage{float}
\usepackage[all]{xy}
\usepackage[normalem]{ulem}
\usepackage[active]{srcltx}
\usepackage{mathrsfs}
\usepackage{etoolbox}
\usepackage[colorlinks=true, breaklinks=true, allcolors=blue]{hyperref}

\usepackage{xspace} %%por causa do espaçao depois dos comandos
\newlength\yStones
\newlength\xStones
\newlength\xxStones

\makeatletter
\def\Stones{\pst@object{Stones}}
\def\Stones@i#1{%
  \pst@killglue%
  \begingroup%
  \use@par%
  \setlength\xxStones{\xStones}%
  \expandafter\Stones@ii#1,,\@nil
  \endgroup
  \global\addtolength\xStones{0.6cm}%
  \global\addtolength\yStones{-7.5mm}}%
\def\Stones@ii#1,#2,#3\@nil{%
  \rput(\xxStones,\yStones){%
    \psframebox[framesep=0]{%
      \parbox[c][6mm][c]{11mm}{\makebox[11mm]{$#1$}}}}%
  \addtolength\xxStones{1.2cm}%
  \ifx\relax#2\relax\else\Stones@ii#2,#3\@nil\fi}
\makeatother

\def\Stone#1{\fbox{\makebox[12mm]{\strut#1}}\kern2pt}

\newcommand{\C}{\kc}

\newcommand{\Z}{\mathbb{Z}}
\newcommand{\kc}{\Bbbk}

\newcommand{\Q}{\mathbb{Q}}
\newcommand{\g}{\mathfrak{g}}
\newcommand{\h}{\mathfrak{h}}

\newcommand{\boh}{\overline{0}}
\newcommand{\bon}{\overline{1}}
\newcommand{\Ann}{\textup{Ann} \,}
\newcommand{\Supp}{\textup{Supp} \,} 
\newcommand{\inj}{\textup{inj} \,} 
\newcommand{\cG}{\mathcal{G}}
\newcommand{\cA}{\mathcal{A}}
\newcommand{\Specm}{\textup{Specm} \,}
\newcommand{\Hom}{\textup{Hom} \,}

\renewcommand{\k}{{\Bbbk}}

\newtheorem{theorem}{Theorem}[section]
\newtheorem{lemma}[theorem]{Lemma}
\newtheorem{corollary}[theorem]{Corollary}
\newtheorem{proposition}[theorem]{Proposition}

\theoremstyle{definition}
\newtheorem{example}[theorem]{Example}
\newtheorem{remark}[theorem]{Remark}

\newtheorem{definition}[theorem]{Definition}
\allowdisplaybreaks
\newtoggle{details}
\iftoggle{details}{
\newcommand{\details}[1]{{\color{blue}\noindent\textbf{Details:}{#1}}}
                  }{
                  \newcommand{\details}[1]{}
                  }
                
\begin{document}
\begin{title}[Harish-Chandra modules for map and affine Lie superalgebras]{Harish-Chandra modules for map and affine Lie superalgebras}
\end{title}

\author[L. Calixto]{Lucas Calixto}
\address{Departamento de Matem\'atica \\ Instituto de Ci\^encias Exatas \\ UFMG \\ Belo Horizonnte \\ Minas Gerais \\ Brazil}
\email{lhcalixto@ufmg.br}

\author[V. Futorny]{Vyacheslav Futorny}
\address{Shenzhen nternational Center for Mathematics, Southern University of Science and Technology, Shenzhen, China}
\email{vfutorny@gmail.com}

\author[H. Rocha]{Henrique Rocha}
\address{Instituto de Matem\'atica e Estat\'istica \\ Universidade de S\~ao Paulo \\ S\~ao Paulo \\ Brazil}
\email{hrocha@ime.usp.br}

%-----------------------------------------------------------------%

\begin{abstract}
We obtain a classification of simple modules with finite weight multiplicities over  basic classical map superalgebras. Any such module is parabolically induced from a simple cuspidal bounded module over a cuspidal map superalgebra. Further on, any simple cuspidal bounded module is isomorphic to an evaluation module. As an application, we obtain a classification of all simple Harish-Chandra modules for basic classical loop superalgebras. Finally, we show that for affine Kac-Moody Lie superalgebras of type I the Kac induction functor reduces the classification of all simple bounded  modules to the classification of the same class of modules over the even part, whose classification is claimed by Dimitrov and Grantcharov. 
\end{abstract}

%-----------------------------------------------------------------%

\subjclass[2010]{17B10, 17B65}
\keywords{affine Lie superalgebras, weight modules, Harish-Chandra modules}

%-----------------------------------------------------------------%

\maketitle

%-----------------------------------------------------------------%

% Sumário
\tableofcontents    % imprime o sumário

%-----------------------------------------------------------------%

%%introdução
\section*{Introduction}

Lie superalgebras can be viewed as generalizations of Lie algebras that arose in the 1970's mainly motivated by the development of supersymmetry. Since then, the study of Lie superalgebras and their representations (or modules) has been an active area of research that has proved to have many connections and applications across several areas of Mathematics and Physics \cite{Var04}. Although the representation theory of Lie algebras and Lie superalgebras have some similarities,  the latter is much harder. Instances of this fact can be seen, for example, in the study of the category of finite-dimensional modules over a finite-dimensional simple Lie superalgebra, which is not semisimple in general \cite{Kac78}; or in the study of the primitive spectrum of a universal enveloping superalgebra \cite{Cou16}.

In this paper, we consider classes of Lie superalgebras that have drawn a lot of attention in recent years. Namely, the classes of \emph{map superalgebras} and \emph{affine (Kac-Moody) Lie superalgebras}. Given an affine scheme of finite type $X$ and a Lie superalgebra $\g$, both defined over the same ground field $\kc$, the map superalgebra associated to $X$ and $\g$, denoted by $M(X,\g)$, is the Lie superalgebra of all regular maps from $X$ to $\g$. It is useful to note that the map superalgebra $M(X,\g)$ is isomorphic to the tensor product $\mathcal{G}:=\g \otimes_{\kc} A$, where $A=\mathcal{O}_X(X)$. Map superalgebras provide a unified way of realizing many important classes of Lie superalgebras, such as generalized current superalgebras, loop and multiloop superalgebras, Krichever-Novikov superalgebras. The affine Lie superalgebra $\mathcal{L}(\g)$ and the affine Kac-Moody Lie superalgebra $\widehat{\g}$ are obtained via certain extensions of the loop superalgebra $M(\C^\times, \g)$.

Representation theory of map algebras is fairly well developed. The classification of all simple  finite-dimensional modules was first obtained for loop algebras by Chari and Pressley \cite{Cha86,CP86}, and then generalized for arbitrary map algebras \cite{NSS12}. In the super setting, a combination of results from \cite{Sav14, CMS16, CM19} provides a classification of all simple finite-dimensional modules over any classical map superalgebra. In both, super and non-super cases, the classification relies on a class of modules called (generalized) evaluation modules, which will be also important in this paper (see Section~\ref{sectioon:cuspidaluniformboundmodules}). Other classes of finite-dimensional modules were also studied for map algebras and superalgebras in \cite{FL04,CFK10,FMS15,BCM19,CLS19}. 

A comprehensive program seeking a description of  simple Harish-Chandra modules (or \emph{finite} weight modules as we call them throughout the text), that is, simple 
weight modules with finite-dimensional weight subspaces over a Lie algebra started at the end of the 1980's with works of Fernando, Britten and Lemire \cite{Fer90,BL87,BL99}. The classification of all such modules over finite-dimensional reductive Lie algebras was then obtained by Mathieu in \cite{Mat00}. Similar approaches were also developed for finite-dimensional and affine Kac-Moody Lie (super)algebras in \cite{DMP00,Fut97,DMP04,RF09,GG20,CF21,You20}, and for map algebras over finite-dimensional reductive Lie algebras in \cite{BLL15, Lau18}. 

In the current paper, we make a significant step in the study of infinite-dimensional modules over map superalgebras and in the classification of simple Harish-Chandra modules over affine (Kac-Moody) Lie superalgebras. Namely, we classify all simple Harish-Chandra $\mathcal{G}$-modules, under the assumption that $\g$ is a basic classical Lie superalgebra. We prove that any such module is either cuspidal bounded, or parabolically induced from a cuspidal bounded module over a cuspidal subalgebra of $\cG$. Moreover, we show that any cuspidal bounded module is isomorphic to an evaluation module. This classification  yields in turn a classification of the same class of modules over the affine Lie superalgebra $\mathcal{L}(\g)$. Hence, our first main result is the following.

\begin{theorem}
\begin{itemize}
\item[(i)] Every  simple Harish-Chandra $\mathcal{G}$-module is parabolically induced from a simple bounded cuspidal module over the Levi factor;
\item[(ii)] If $\g_{\boh}$ is semisimple (in particular, this is the case when $\cG$ admits cuspidal modules), then 
 any simple cuspidal bounded $\mathcal{G}$-module is isomorphic to an evaluation module.
 \end{itemize}
\end{theorem}

After aplying these results to affine Lie algebras, we study bounded modules over the affine Kac-Moody Lie superalgebras associated to a basic classical Lie superalgebras of type I. It is a well known fact that all simple finite-dimensional modules over $\g$ can be obtained as simple quotients of Kac modules. Moreover, it was shown in \cite{RF09} that in some cases any simple Harish-Chandra module of $\widehat{\g}$ of nonzero level is parabolically induced from a cuspidal simple module over a subalgebra of even part of $\widehat{\g}$, e.g. in the case of $\widehat{\mathfrak{sl}(n|m)}$. This phenomenon was studied  in \cite{CM21} for general superalgebras with finite-dimensional odd part, in which case simple modules are obtained as simple quotients of Kac modules. For
basic classical Lie superalgebras of type I, 
we use the Kac functor to reduce the classification of all bounded simple $\widehat{\g}$-modules of level zero to the classification of the same class of modules over $\widehat{\g}_0$. Namely, we have the second main result.

\begin{theorem} Let $\g$ be a basic classical Lie superalgebras of type I. Then the Kac induction functor 
gives  a bijection between the sets of isomorphism classes of simple bounded Harish-Chandra $\widehat{\g}$-modules and simple bounded Harish-Chandra
 $\widehat{\g}_0$-modules. 
\end{theorem}

This paper is organized as follows. In Section~\ref{section:preliminaries} we fix notation and state basic results that will be needed in the subsequent sections. We prove some results regarding tensor product of infinite-dimensional representations in Section~\ref{section:tensorproductheorem}. We stress that these results are quite interesting by their own, as they generalize (to the infinite-dimensional super setting) classical statements which we could not find in the existing literature in the generality we need. In Section~\ref{section:paraboolicinductiontheorem} we prove a parabolic induction theorem for finite weight modules over map superalgebras associated to basic classical Lie superalgebras. The proof is based on the concept of the shadow of a module, which was introduced by Dimitrov, Mathieu and Penkov in \cite{DMP00}, and reduces the classification problem to the problem of classifying all cuspidal bounded modules over cuspidal subalgebras of $\cG$. In Section~\ref{sectioon:cuspidaluniformboundmodules}, we classify all cuspidal bounded modules over cuspidal map superalgebras in terms of evaluation modules. Section~\ref{section:applicationtoaffineliesuperalgebras} is devoted the case of affine (Kac-Moody) Lie superalgebras. Firstly, we apply our results for the particular case of loop superalgebras $M(\C^\times, \g)$ to obtain a description of all finite weight modules over $\mathcal{L}(\g)$. 
Then, we assume that $\g$ is a basic classical Lie superalgebras of type I 
in order to define analogs of Kac $\widehat{\g}$-modules. 
The main result of this section generalizes to the affine setting the reduction results from \cite{RF09} and \cite{CM21} for bounded modules.
\section{Preliminaries}
\label{section:preliminaries}

 In this paper, we fix an uncountable algebraically closed field $\kc$ with characteristic $0$. For a super vector space $V=V_{\boh} \oplus V_{\bon}$, we denote by $|v|$ the parity of a homomogeneous element $v \in V_{i}$, $i \in \Z_2$. The set of linear homomorphisms between two super vector spaces $V$ and $W$ is itself a super vector space, and we will denote it by $\Hom_{\kc}(V,W)$. For any Lie superalgebra $\mathfrak{a}$ we denote by $U(\mathfrak{a})$ its universal enveloping superalgebra. An $\mathfrak{a}$-module is by definition a module over the superalgebra $U(\mathfrak{a})$, hence an $\mathfrak{a}$-module is a super vector space $V = V_{\boh} \oplus V_{\bon}$ such that $\mathfrak{a}_{i} V_{j} \subset V_{i+j}$ for each $i,j \in \Z_2$. A linear map $\phi:V \rightarrow W$ between $\mathfrak{a}$-modules $V,W$ is a homomorphism of $\mathfrak{a}$-modules if $\phi(x v) = (-1)^{|\phi| |x| }x \phi(v)$. Note that we also consider odd homomorphisms between modules and super vector spaces. The set of all $\mathfrak{a}$-module homomorphisms between $V$ and $W$ is denoted by $\mathrm{Hom}_{\mathfrak{a}}(V,W)$, and when $V=W$ it is denoted by $\mathrm{End}_{\mathfrak{a}}(W)$.

 Let $\g = \g_{\boh} \oplus \g_{\bon}$ be a basic classical Lie superalgebra with Cartan subalgebra $\h \subset \g_{\boh}$ (see \cite{Kac77}). Let $A$ be a finitely generated associative commutative algebra with identity over $\kc$, and set $\cG=\g \otimes A$ the map superalgebra associated to $\g$ and $A$. A $\cG$-module $V$ is called a weight module if $V$ is a weight module over $\g\cong \g\otimes \k\subseteq \cG$, that is, if we have a decomposition $V= \bigoplus_{\lambda\in \h^*}V^\lambda$, where $V^\lambda = \{v\in V\mid hv=\lambda(h)v\text{ for all } h\in \h\}$ is the weight space associated to the weight $\lambda$. The support of $V$ is the set $\Supp V = \{\lambda\in \h^*\mid V^\lambda\neq 0\}$. Throughout this article a $\cG$-module is always assumed to be a weight module. A finite weight module is a weight module whose dimensions of all its weight spaces are finite. A finite weight module $V$ is said to be bounded if there is $n\gg 0$ for which $\dim V^\lambda \leq n$ for all $\lambda\in \Supp V$. Notice that $\cG$ is a weight module under the adjoint action, and $\cG^\alpha = \g^\alpha\otimes A$ for all $\alpha\in \Supp \cG$. Thus $\cG$ is a finite weight module if and only if $\dim A<\infty$. Moreover, the root system of $\cG$ is the same as that of $\g$ and it is given by $\Delta := \Supp \cG \setminus \{0\}$. In particular, we have that $\Delta = \Delta_{\boh} \cup \Delta_{\bon}$, and every root of $\cG$ is either in $\Delta_{\boh}$ or in $\Delta_{\bon}$. For each $\alpha \in \Delta$, we have that $\dim \g^{\alpha} = 1$, and hence we will fix a nonzero vector $x_{\alpha} \in \g^{\alpha}$, and we set $h_{\alpha} := [x_{\alpha},x_{-\alpha}]$. When $\alpha \in \Delta_{\boh}$, we will always assume $\left \{x_{\alpha},x_{-\alpha},h_{\alpha} \right\}$ is a $\mathfrak{sl}_2$-triple. We define $\mathcal{B}= \{ x_{\alpha},h_{\alpha} \mid \alpha \in \Delta \}\subseteq \g$, which generates $\g$ as a vector space. We will denote by $\mathcal{Q} = \Z \Delta$ (resp. $\mathcal{Q}_{\boh} = \Z \Delta_{\boh}$) the lattice generated by $\Delta$ (resp. $\Delta_{\boh}$).

Recall that modules over $\g_{\boh}$ on $\g_{\bon}$ are either simple, in which case $\g$ is said to be of \emph{type II}, or it is a direct sum of two simple modules, and we say $\g$ is of \emph{type I}. In Table \ref{table:simplesliesuperalgberas} we list all basic classical Lie superalgebra that are not Lie algebras, together with their even part and their type. Note that $\g_{\boh}$ is either semisimple or a reductive Lie algebra with one-dimensional center. Any basic classical Lie superalgebra $\g$ admits a \emph{distinguished $\Z$-grading} $\g=\bigoplus_{n \in \Z} \g_n$, which is compatible with the $\Z_2$-grading and satisfies:
\begin{enumerate}
    \item if $\g$ is of type I, then $\g_{\boh} = \g_0$, $\g_{\bon} = \g_{-1} \oplus \g_1$, and
    \item if $\g$ is of type I, then $\g_{\boh} = \g_{-2} \oplus \g_0 \oplus \g_{2}$, $\g_{\bon} = \g_{-1} \oplus \g_1$.
\end{enumerate}

\begin{table}
 \begin{tabular}{|l   l  l |} 
 \hline
 $\g$ & $\g_{\boh}$ & Type \\ [0.5ex] 
 \hline\hline
 $A(m,n)$, $m>n\geq 0$ &  $A_m \oplus A_n \oplus \kc$ & I  \\ 
 \hline
 $A(n,n)$, $n \geq 1$ & $A_n \oplus A_n$ & I  \\
 \hline
 $C(n+1)$, $n \geq 1$ & $C_n \oplus \kc$ & I  \\
 \hline
 $B(m,n)$, $m\geq 0$, $n \geq 1$ & $B_m \oplus C_n$ & II  \\
 \hline
 $D(m,n)$, $m\geq 2$, $n\geq 1$ & $D_m \oplus C_n$ & II  \\ 
 \hline
 $F(4)$ & $A_1 \oplus B_3$ & II  \\
 \hline
 $G(3)$ & $A_1 \oplus G_2$ & II  \\
 \hline
 $D(2,1,a)$, $a \neq 0,-1$ & $A_1 \oplus A_1 \oplus A_1$ & II  \\
 \hline
\end{tabular}
\vspace{0.5cm}
\caption{Basic classical Lie
superalgebras that are not Lie
algebras, their even
part and their type}
\label{table:simplesliesuperalgberas}
\end{table}

The analogue for Lie superalgebras of the Poincaré-Birkhoff-Witt Theorem (PBW Theorem) is formulated in the following way (see \cite[Chapter §2, Section 3]{Sch79}).
\begin{lemma}
   Let $\mathfrak{a}$ be a Lie superalgebra over $\kc$. Let $B_0, B_1$ be totally ordered bases of $\mathfrak{a}_{\boh},\mathfrak{a}_{\bon}$, respectively. Then the standard monomials
   $$u_1\dots u_r v_1 \dots v_s, \ u_1,\dots,u_r \in B_0, \ v_1,\dots,v_s \in B_1, u_1 \leq \dots \leq u_r, \ v_1<\dots <v_s ,$$
   form a basis for the universal enveloping superalgebra $U(\mathfrak{a})$. In particular, we have an isomorphism of vector spaces
    $$U(\mathfrak{a}) \cong U \left (\mathfrak{a}_{\boh} \right ) \otimes \bigwedge \left (\mathfrak{a}_{\bon} \right ).$$
\end{lemma}

For any $\mathcal{G}$-module $V$, define
$$\Ann_A(V) := \left \{ a \in A \mid (\g \otimes a)V = 0  \right \}.$$

The following results reduces the study of simple finite weight modules over an arbitrary map superalgebra to the study of simple finite weight modules over certain finite-dimensional map superalgebras.

\begin{proposition}{\cite[Proposition~8.1]{Sav14}}
    If $\mathcal{I}$ an ideal of $\mathcal{G}$, then exists an ideal $I$ of $A$ such that $\mathcal{I}=\g \otimes I$.
\end{proposition}

The proof of the following proposition is similar to the Lie algebra case. See \cite[Proposition 4.3]{BLL15}.
\begin{proposition}\label{proposition:idealcofinite}
    If $V$ is a simple finite weight $\mathcal{G}$-module, then $A/\Ann_A(V)$ is a finite-dimensional algebra.
\end{proposition}

\section{A tensor product theorem}
\label{section:tensorproductheorem}

This section is devoted to prove an infinite-dimensional version of a classical result proved by Cheng \cite{Che95}. Namely, that if $V_i$ is a simple module over a Lie superalgebra $\g_i$, for $i=1,2$, then the $(\g_1\oplus \g_2)$-module $V_1\otimes V_2$ is either simple or there is a simple submodule $V\subseteq V_1\oplus V_2$ for which $V_1\otimes V_2\cong V\oplus V$. We start with a generalization of Schur's lemma to the super infinite-dimensional setting. 

\begin{lemma}[Schur's Lemma]\label{lemma:lemmadeschur}
Let $\g$ be a Lie superalgebra and $V$ be a simple $\g$-module. Then one of the following statements hold:
\begin{enumerate} 
    \item $\textup{End}_{\g} (V) = \textup{End}_{\g} (V)_{\boh} = \kc id$,
    
    \item $\textup{End}_{\g}(V)= \kc id \oplus \kc \sigma$, where $\sigma^2=id$ is an odd element that permutes $V_{\boh}$ and $V_{\bon}$. In particular, $\sigma$ provides an isomorphism of $\g_{\boh}$-modules between $V_{\boh}$ and $V_{\bon}$.
\end{enumerate}
\end{lemma}
\begin{proof}
       We know that $\textup{End}_{\g} (V) = \textup{End}_{\g} (V)_{\boh} \oplus \textup{End}_{\g} (V)_{\bon}$, where $\textup{End}_{\g} (V)_{\boh} = \kc id$ by the infinite-dimensional Schur's Lemma (due to Dixmier) see \cite[Proposition 2.6.5, Corollary 2.6.6]{Dix96}. If $\textup{End}_{\g} (V)_{\bon} = 0$, then we have case $(1)$. Suppose $\textup{End}_{\g} (V)_{\bon} \neq 0$, and take a nonzero element $\sigma \in \textup{End}_{\g} (V)_{\bon}$. Since $\textup{ker} \, \sigma$ is a proper $\g$-submodule of $V$, we see that $\textup{ker} \, \sigma = 0$. On the other hand, as $\textup{im} \, \sigma $ is a nonzero $\g$-submodule of $V$, we have that $\textup{im} \, \sigma = V$.  In particular, for $\sigma^2 \in \textup{End}_{\g} (V)_{\boh}$ we obtain $\sigma^2= K id$ for some nonzero $K \in \kc$. We can assume $K=1$. Now let $\tau \in \textup{End}_{\g} (V)_{\bon}$ be a nonzero element. As before we get that $\sigma \circ \tau = K_1 id$ for a nonzero $K_1 \in \kc$, which implies $\sigma^2 \circ \tau =id \circ \tau=\tau = K_1 \sigma \circ id = K_1 \sigma$.
\end{proof}

\begin{lemma}\label{lemma:densitytheoremversion2}
Let $\g$ be a Lie superalgebra, and let $V$ be a simple $\g$-module.
\begin{enumerate}
    \item Assume that $\textup{End}_{\g}(V)_{\bon} =  0$, that $v_1,\dots,v_n \in V$ are linearly independent vectors and that $w_1,\dots,w_n \in V$ are arbitrary elements. Then there exists an element $u \in U(\g)$ such that $u v_i=w_i$.
    
    \item Assume that $\textup{End}_{\g}(V)_{\bon}\neq 0$, that $v_1,\dots,v_n \in V_{i}$ are linearly independent vectors and that $w_1,\dots,w_n \in V_i$ are arbitrary elements, where $i \in \{ \boh, \bon \}$. Then there exists an even element $u \in U(\g)_{\boh}$ such that $u v_{i}=w_i$.
\end{enumerate}
\end{lemma}
\begin{proof}
Both statements follow from the Jacobson Density Theorem and Schur's Lemma (see \cite[Proposition 8.2]{Che95}). Also observe that $u$ is even in (2) since all $v_i$'s and all $w_i$'s have the same parity.
\end{proof}

\begin{lemma}\label{proposition:tensorissimple}
Let $\g_1$ and $\g_2$ be Lie superalgebras, and $V_1$, $V_2$ be simple modules over $\g_1$ and $\g_2$, respectively. If $\textup{End}_{\g_1} \, V_1  \cong \kc$, then $V_1 \otimes V_2$ is a simple $\g_1 \oplus \g_2$-module.
\end{lemma}
\begin{proof}
We need to show that for any nonzero $v \in V_1 \otimes V_2$ we have that $U(\g_1 \oplus \g_2) v = V_1 \otimes V_2$. Let $v \in V_1 \otimes V_2$ be an arbitrary homogeneous nonzero element of $V_1 \otimes V_2$. Then we can write $v= \sum_{j=1}^r v^1_j \otimes v_j^2.$ 
  Suppose, without loss of generality, that $\{v_j^1\}_{j=1}^r$ is a $\kc$-linearly independent set and $v_1^2  \neq 0$. By Lemma \ref{lemma:densitytheoremversion2} (1), there is $u \in U(\g_1)$ such that $uv_j^1 =\delta_{j1} v_1^1 $. Thus, $u \in U(\g_1 \oplus \g_2 ) $ and 
  $$ u v = \sum_{j=1}^r (uv_j^1) \otimes v_j^2 = v_1^1 \otimes v_1^2 \neq 0.$$

  Let $w_1 \in V_1$ and $w_2 \in V_2$ be arbitrary elements. By the irreducibly of $V_1$ and $V_2$ as modules over $\g_1$ and $\g_2$, respectively, there exists $u_1 \in U(\g_1)$ and $u_2 \in U(\g_2)$ such that $u_1 v_1^1 = w_1$ and $u_2 v_1^2 = w_2$. Thus,
	\begin{align*}
			u_1 ((u_2 (u v))) & = u_1 (u_2(v_1^1 \otimes v_1^2 )) \\
			&= u_1((-1)^{|u_2| | v_1^1|}v_1^1 \otimes( u_2 v_1^2)) \\
			&= (-1)^{|u_2| | v_1^1|} u_1(v_1^1 \otimes w_2) \\
			& =(-1)^{|u_2| | v_1^1|}(u_1 v_1^1 )\otimes w_2 = (-1)^{|u_2| | v_1^1|}  w_1 \otimes w_2.
	\end{align*}
  Therefore, $w_1 \otimes w_2 \in U(\g_1 \oplus \g_2 ) v$ for all $w_1 \otimes w_2 \in V_1 \otimes V_2$, which implies that $V_1 \otimes V_2 = U(\g_1 \oplus \g_2) v$.
  \end{proof}

%--------------------------------------------------------------------------%

\begin{proposition}\label{proposition:chengresultfortensorproduct}
Let $\g_1$ and $\g_2$ be Lie superalgebras, and $V_1$, $V_2$ be simple modules over $\g_1$ and $\g_2$, respectively. Then $V_1 \otimes V_2$ is either a simple $\g_1 \oplus \g_2$-module, or it is isomorphic to $V \oplus V$, where $V$ is a simple $\g_1 \oplus \g_2$-module.
\end{proposition}
\begin{proof}

    If $\textup{End}_{\g_1}(V_1) \cong \kc$ or $\textup{End}_{\g_2}(V_2) \cong \kc$, then by Proposition \ref{proposition:tensorissimple} $V_1 \otimes V_2$ is a simple $\g_1 \oplus \g_2$-module. By Lemma \ref{lemma:lemmadeschur} we can assume that $\textup{End}_{\g_1}(V_1) \cong \kc id \oplus \kc \sigma_1$ and $\textup{End}_{\g_1}(V_1) \cong \kc id \oplus \kc \sigma_2$, where $\sigma_1$ and $\sigma_2$ are odd elements such that $\sigma_1^2=id$ and $\sigma_2^2=-id$.
    
    The map $\sigma(v_1 \otimes v_2) =  (-1)^{|v_1|} \sigma_1(v_1) \otimes \sigma_2(v_2)$ is a $\g_1 \oplus \g_2$-module endomorphism, and $\sigma^2 =id$.

    Note that for every $x \in V_1 \otimes V_2$ 
    $$x = \left (\frac{x - \sigma(x)}{2} \right) + \left( \frac{x+\sigma(x)}{2} \right),$$
    thus $V_1 \otimes V_2 = V \oplus V'$ with
    \begin{align*}
        V = \{x \in V_1 \otimes V_2 \mid \sigma(x)=x \} \quad \text{and} \quad
        V' = \{x \in V_1 \otimes V_2 \mid \sigma(x)=-x \}.
    \end{align*}
     Note that $V$ and $V'$ are $\g_1 \oplus \g_2$-submodules of $V_1 \otimes V_2$. Let $\{w_i \in V_2 \mid i \in I\}$ be a basis of the even part $V_2$, then $\{w_i, \ \sigma_2(w_i) \mid i \in I \}$ is a basis of $V_2$. It is possible to show that $V$ is generated by 
    $\{v \otimes w_j + \sigma(v \otimes w_j) \mid v \in V_1, \ j \in I\}$, and $V'$ is generated by $\{v \otimes w_j - \sigma(v \otimes w_j) \mid v \in V_1, \ j \in I\}$.

The automorphism of $\g_1 \oplus \g_2$-modules
$\sigma_1 \otimes 1 :V_1 \otimes V_2 \rightarrow V_1 \otimes V_2$
sends $V$ to $V'$ and $V'$ to $V$, thus $V \cong V'$.

    Let $v \in V$ be a homogeneous element of $V$, then there exists $v_1,\dots,v_n$ such that 
    \[
    v = \sum_{j=1}^n v_j \otimes w_{i_j} + \sigma(v_j \otimes w_{i_j})
    \]
    with $v_1 \neq 0$ and $i_k \neq i_l$ if $k \neq l$. By Lemma \ref{lemma:densitytheoremversion2}, there exist $u \in U(\g_2)$ such that $uw_{i_j} = \delta_{1j} w_{i_1}$. Hence, $uv = v_1 \otimes w_{i_1}+ \sigma (v_1 \otimes w_{i_1})$. If $v_0 \in V_1$ is homogeneous and $k\in I$, then there exists $a\in U(\g_1)$ and $b\in U(\g_2)$ such that $a v_1 = v_0$ and $bw_{j_1} = b w_k$ because both $V_1$ and $V_2$ are simple. Thus,
    \[
     a(b(uv)) = v_0 \otimes w_k + \sigma(v_0 \otimes w_k),
    \]
    and the generating set $\{v_1 \otimes w_j + \sigma(v_1\otimes w_j) \mid v_1 \in V_1, \ j \in I \}$ of $V$ is a subset of $U(\g_1 \otimes \g_2)v$. We conclude that $V$ is simple because $V=U(\g_1 \otimes \g_2)v$ for every nonzero homogeneous element $v\in V$.
    
\end{proof}

\begin{definition}
Using the notation of Proposition \ref{proposition:chengresultfortensorproduct}, we define the \emph{irreducible tensor product} of $V_1$ and $V_2$ as 
    $$
    V_1 \hat{\otimes} V_2 =
    \begin{cases}
        V_1 \otimes V_2, \quad & \mbox{if} \ V_1 \otimes V_2 \text{ is simple}, \\
        V, \quad & \mbox{if} \ V_1 \otimes V_2 \text{ is not simple,}
    \end{cases}
    $$
    where $V$ is the (unique) simple submodule of $V_1\otimes V_2$, obtained in the proof of Proposition \ref{proposition:chengresultfortensorproduct}, for which we have an isomorphism of $(\g_1\oplus \g_2)$-modules  $V_1 \otimes V_2\cong V\oplus V$.
\end{definition}

\begin{lemma}\label{lemma:simpleofsumistensor}
    Let $\g$ be a basic classical Lie superalgebra and $V$ be a finite weight $\g$-module. If there is $\lambda \in \h^*$ such that $W^{\lambda} = \{w \in W \mid h w=\lambda(h)w \ \text{for all} \ h \in \h \}$ is nonzero for all submodules $W \subset V$, then $V$ contains a simple $\g$-module. 
\end{lemma}
\begin{proof}
The proof for the Lie algebra case works in our setting \cite[Lemma~3.3]{BLL15}.

\end{proof}

\begin{theorem}\label{proposition:simpleofsumistensor}
Let $\g_1$ and $\g_2$ be basic classical Lie superalgebras, and $S_1,S_2$ be commutative associative unital algebras. If $V$ is a simple finite weight $\left(\g_1 \otimes S_1 \oplus \g_2 \otimes S_2\right)$-module, then $V \cong V_1 \hat{\otimes} V_2$ where $V_1$ and $V_2$ are simple finite weight modules over $\g_1 \otimes S_1$ and $\g_2 \otimes S_2$, respectively. Moreover, if $\textup{End}_{\g_i \otimes S_i} (V_i) \cong \kc$ for some $i=1,2$, then $V \cong V_1 \otimes V_2$.
\end{theorem}

\begin{proof}
The proof is similar to that of \cite[Proposition~3.4]{BLL15} with some slight modifications. Let $v \in V^{(\lambda,\mu)}$ be a nonzero vector of weight $(\lambda,\mu) \in \h_1^* \times \h_2^*$. We have that 
$$h_1 u v = (-1)^{|u| |h_1| } u h_1 v + [h_1,u]v = \lambda(h_1)u v \text{ for all } u \in U(\g_2 \otimes S_2), \ h_1 \in \h_1^*. $$ 
Then $W= U(\g_2 \otimes S_2)v \subset V$ is a finite weight module for $\g_2 \otimes S_2$, because $W^{\eta} \subset V^{(\lambda,\eta)}$.

Let $N$ be any nonzero $(\g_2 \otimes S_2)$-submodule of $W$. Then define $H^N$ to be the subspace of $\Hom_{\kc}(N,V)$ generated by all homogeneous elements $\varphi\in \Hom_{\kc}(N,V)$ such that $y\varphi(w)=(-1)^{|\varphi||y|}\varphi(yw)$ for all $ y \in \g_2 \otimes S_2$, $w \in N$. Notice that $H^N$ is a nonzero vector space and it is a module over $\g_1 \otimes S_1$, with action $(x \varphi)(w)=x(\varphi(w))$, for all $x \in \g_1 \otimes S_1$, $\varphi \in H^N$, and $w \in N$. 

    Let $M \subset H^N$ be a nonzero $(\g_1 \otimes S_1)$-submodule. The map 
    \begin{align*}
          \Psi_{M,N} : M \otimes N & \rightarrow V \\
          \varphi \otimes w & \mapsto \varphi(w)
    \end{align*}
    is a nonzero $(\g_1 \otimes S_1) \oplus (\g_2 \otimes S_1)$-module homomorphism. By the simplicity of $V$, $\Psi_{M,N}$ is surjective. Note that $(M \otimes N)^{(\alpha,\beta)} = M^{\alpha} \otimes N^{\beta}$, therefore $\Psi_{M,N}$ restricts to a surjection $M^{\lambda} \otimes N^{\mu} \rightarrow V^{(\lambda,\mu)}$. Thus, $M^{\lambda}$ and $N^{\mu}$ are nonzero subspaces, since $V^{(\lambda,\mu)} \neq 0$.

    Since $W$ is a finite weight $(\g_2 \otimes S_2)$-module, it has a simple $(\g_2 \otimes S_2)$-submodule $Q \subset W$ such that $Q^{\mu} \neq 0$, by Lemma \ref{lemma:simpleofsumistensor}. Let $w \in Q^{\mu}$ be a nonzero weight vector. Then $Q = U(\g_2 \otimes S_2) w$ by its simplicity. If $\varphi \in H^Q$ is homogeneous, then $\varphi(uw)=(-1)^{|u||\varphi|}u\varphi(w)$ for all $u \in U(\g_2 \otimes S_2)$. Thus every element of $H^Q$ is completely defined by its value on $w$. Using this and the fact that every weight space of $V$ has finite dimension, one can show that $H^Q$ is a finite weight $(\g_1 \otimes S_1)$-module. Since for every submodule $L \subset H^Q $ we have that $L^{\lambda} \neq 0$, then $H^Q$ has a simple finite weight submodule $P$ over $\g_1 \otimes S_1$, by Lemma \ref{lemma:simpleofsumistensor}.  By Proposition \ref{proposition:chengresultfortensorproduct}, the $(\g_1 \otimes S_1) \oplus (\g_2 \otimes S_2)$-module $P \otimes Q$ is either simple (this is the case when $\textup{End}_{\g_1 \otimes S_1} (P) \cong \kc$ or $\textup{End}_{\g_2 \otimes S_2} (Q) \cong \kc$) or it contains a simple submodule $K$ for which $P \otimes Q \cong K \oplus K$. If $P \otimes Q$ is simple, then $\Psi_{P,Q}$ is an isomorphism. If $P \otimes Q \cong K \oplus K$, then every nonzero proper submodule (resp.  quotient) of $P\otimes Q$ is isomorphic to $K$. In particular, $\Psi_{P,Q}$ induces an isomorphism between $K$ and $V$, and hence $V \cong P \hat{\otimes} Q$.
\end{proof}

\begin{remark}
We point out that both Lemma~\ref{lemma:simpleofsumistensor} and Theorem~\ref{proposition:simpleofsumistensor} still hold for a more general class of Lie superalgebras. Namely, the class of all Lie superalgebras $\g$ that admit an abelian subalgebra $\h\subseteq \g$ acting (via the adjoint action) semisimply on $\g$. In this case, we would consider the class of all $\g$-modules on which $\h$ acts semisimply. In particular, this is the case for any abelian Lie algebra.
\end{remark}

%--------------------------------------------------------------------------%

\section{Parabolic induction theorem}
\label{section:paraboolicinductiontheorem}
In this section we prove a parabolic induction theorem, which reduces the classification of simple finite-weight modules over $\cG$ to the classification of the so-called cuspidal modules over cuspidal Lie subalgebras of $\cG$.

\begin{lemma}\label{lemma:actseverywhere02}
    Let $V$ be a simple finite weight $\mathcal{G}$-module. Let $\alpha \in \Delta$ and  $a \in A$ a non zero element. Then either $x_{\alpha} \otimes a $ acts locally nilpotently everywhere on $V$ or it acts injectively everywhere on $V$. 
\end{lemma}
\begin{proof}
    Similarly to the Lie algebra case, this result follows from the fact that the set of vectors of $V$ such that $x_{\alpha} \otimes a$ acts locally nilpotently is a submodule of $V$, along with the fact that every element $x_{\beta}$ with $\beta \in \Delta$ acts nilpotently on $\g$ via the adjoint representation.
\end{proof}

\begin{proposition}\label{proposition:equivalentnilpotent}
    Let $V$ be a simple finite weight $\mathcal{G}$-module. Suppose $\alpha \in \Delta_{\boh}$ or $\alpha \in \Delta_{\bon}$ with $2\alpha \in \Delta$. Then the following conditions are equivalent
    \begin{enumerate}
        \item\label{propitem:equivalentnilpotent01} For each $\lambda \in \Supp V$, $V^{\lambda+n\alpha}$ is zero for all but finite many $n>0$.
        \item\label{propitem:equivalentnilpotent02} There is $\lambda \in \Supp V$ such that $V^{\lambda+n\alpha}$ is zero for all but finite many $n>0$.
        \item\label{propitem:equivalentnilpotent03} For all $a \in A$, $x_{\alpha} \otimes a$ acts nilpotently on $V$.
        \item\label{propitem:equivalentnilpotent04} $x_{\alpha} \otimes 1$ acts nilpotently on $V$.
    \end{enumerate}
\end{proposition}
\begin{proof}
    It is clear that \eqref{propitem:equivalentnilpotent01} implies \eqref{propitem:equivalentnilpotent02}, \eqref{propitem:equivalentnilpotent03} implies \eqref{propitem:equivalentnilpotent04}. If \eqref{propitem:equivalentnilpotent02} is true, then $x_{\alpha} \otimes a$ acts nilpotently on $V^{\lambda}$. By Lemma \ref{lemma:actseverywhere02}, $x_{\alpha} \otimes a$ acts nilpotently everywhere on $V$. 
    Now suppose \eqref{propitem:equivalentnilpotent04} is true. If $\alpha \in \Delta$ is even, we can use the same argument of \cite[Proposition 2.2]{Lau18} in the map Lie algebra case. Assume $\alpha \in \Delta_{\bon}$ with $2\alpha \in \Delta_{\boh}$. Suppose there is infinitely many $n$ such that $V^{\lambda +n \alpha} \neq 0$. Since $2(x_{\alpha} \otimes 1)^2 = x_{2\alpha} \otimes 1$, we have that $x_{\alpha} \otimes 1$ acts nilpotently on $V$ if and only if $x_{2\alpha} \otimes 1$ acts nilpotently on $V$. However, $V^{\lambda + 2n\alpha} \neq 0$ or $V^{\lambda + \alpha + 2n\alpha} \neq 0$ for infinitely many $n \geq 0$, which contradicts the fact $x_{2\alpha} \otimes 1$ acts nilpotently on $V$ (notice that $2\alpha$ is an even root and we already proved the statement for this case). Therefore, $V^{\lambda+n\alpha}$ is zero for all but finite many $n>0$.
\end{proof}

    \begin{definition}
        Let $\alpha \in \Delta$. We say that $\alpha$ is \emph{locally finite} on $V$ if one, hence all, of conditions of Proposition \ref{proposition:equivalentnilpotent} holds. Similarly, we say that $\alpha$ is \emph{injective} if $x_{\alpha} \otimes 1$ acts injectively on $V$, and we denote the set of all injective roots on $V$ by $\inj V$. 
    \end{definition}

     We say a subset $R \subset \Delta$ is \emph{closed} if $\alpha + \beta \in R$ whenever $\alpha, \beta \in R$ and $\alpha + \beta \in \Delta$. For each $R \subset \Delta$ we define $-R = \{ -\alpha \mid \alpha \in R\}$. A closed set $R \subset \Delta$ is called a \emph{parabolic set} if $R \cup -R= \Delta$.
    \begin{lemma}\label{lemma:rootproperties00}
    If $V$ is a simple finite weight $\mathcal{G}$-module, then $\inj V$ is closed.
    \end{lemma}
    \begin{proof}
    Let $\alpha,\beta \in \inj V$ such that $\alpha+\beta \in \Delta$. By Lemma \ref{lemma:actseverywhere02} and Proposition \ref{proposition:equivalentnilpotent}, $x_{\alpha} \otimes 1$ and $x_{\beta} \otimes 1$ acts injectively everywhere on $V$. Thus $(x_{\alpha} \otimes 1)(x_{\beta}\otimes 1)$ acts injectively, so $V^{\lambda +n(\alpha + \beta)}$ is non zero for some $\lambda \in \Supp V$ and all $n \geq 0$. Applying Lemma \ref{lemma:actseverywhere02} and Proposition \ref{proposition:equivalentnilpotent}, $x_{\alpha +\beta} \otimes 1$ acts injectively on $V$.
    \end{proof}
    
    Following \cite{DMP00}, we define the shadow of $V$ as follows. First, we define $C_V$ as the saturation of the monoid $\Z_+ \inj V\subseteq \mathcal{Q}$, i.e. $$C_V = \left \{ \lambda \in \mathcal{Q} \mid m \lambda \in \Z_+ \inj V  \text{ for some }m > 0 \right \}.$$ Secondly, we decompose $\Delta$ into four disjoint sets $\Delta_V^i$, $\Delta_V^f$, $\Delta_V^+$, $\Delta_V^-$ defined by
\begin{align*}
\Delta_V^i & = \left \{\alpha \in \Delta \mid \pm \alpha \in C_V \right\}, \\
\Delta_V^f & = \left \{\alpha \in \Delta \mid \pm \alpha \notin C_V  \right\}, \\
\Delta_V^+ & = \left \{\alpha \in \Delta \mid \alpha \notin C_V , \ -\alpha \in C_V  \right\}, \\
\Delta_V^- & = \left \{\alpha \in \Delta \mid \alpha \in C_V , \ -\alpha \notin C_V  \right\}.
\end{align*}
Finally, we define the following subspaces of $\g$
    $$ \g_V^+ = \bigoplus_{\alpha \in \Delta_V^+} \g^\alpha, \quad \g_V^i = \h \oplus \bigoplus_{\alpha \in \Delta_V^i} \g^{\alpha}, \quad \g_V^- = \bigoplus_{\alpha \in \Delta_V^-} \g^{\alpha}.$$
    The triple $(\g_V^-, \g_V^0, \g_V^+)$ is called the \emph{shadow} of $V$. 
        \begin{remark}
        Let $V$ be a simple weight $\cG$-module.
\begin{enumerate}
\item If $\alpha,\beta\in C_V$, then $\alpha+\beta\in C_V$.
\item If $\alpha\in C_V$, then there is $m\in \Z_+$ such that $\widetilde{\alpha}:=m\alpha\in \Z_+ \inj V$ satisfies the following condition: for any weight $\lambda\in \Supp V$, we have that $V^{\lambda+n\widetilde{\alpha}}\ne \{0\}$ for all $n\geq 0$. 
\item The following observation is an important difference between super and non super setting: the set $\Delta_V^i$ is not necessarily equal to $\inj V$. Indeed, since the Lie superalgebra $\g = D(2,1,a)$ admits cuspidal modules \cite{DMP00}, we can choose $V$ such that $\Delta_V^i = \Delta$. Notice however that there is no odd root in $\inj V$, as all roots in $\Delta_{\bon}$ are locally finite  ($2\alpha\notin \Delta$ for any $\alpha\in \Delta_{\bon}$). In this case, we actually have that $\Delta_{\boh} \subseteq \inj V$, and $2(\pm \varepsilon_1 \pm \varepsilon_2 \pm \varepsilon_3) = (\pm 2\varepsilon_1) + (\pm 2\varepsilon_2) +(\pm 2\varepsilon_3)\in \Z_+ \Delta_{\boh}\subseteq C_V$ (we are using the same notation as in \cite[pg.~52]{Kac77}).
\end{enumerate}
\end{remark}

Let $V$ be a simple weight $\cG$-module. For $\alpha \in \Delta$ and $\lambda \in \Supp V$,  we define the \emph{$\alpha$-string through $\lambda$} to be the set $\left \{x \in \Q \mid \lambda + \alpha \in \Supp V \right \}$. The following lemma describes the shadow of $V$ in terms of strings.

\begin{lemma}\label{lemma:dasalphastring}
\begin{enumerate}
    \item $\alpha \in \Delta_V^f$ if and only if the $\alpha$-string through any $\lambda \in  \Supp V$ is bounded.
    \item $\alpha \in \Delta_V^i$ if and only if the $\alpha$-string through any $\lambda \in  \Supp V$ is unbounded in both directions.
    \item $\alpha \in \Delta_V^+$ if and only if the $\alpha$-string through any $\lambda \in  \Supp V$ is bounded from above only.
    \item $\alpha \in \Delta_V^-$ if and only if the $\alpha$-string through any $\lambda \in  \Supp V$ is bounded from below only.
\end{enumerate}
\end{lemma}

Let $V$ be a simple finite weight $\mathcal{G}$-module. Then, it follows from Lemma \ref{lemma:rootproperties00} that  $\g_V^+$, $\g_V^i$, and $\g_V^-$ are subalgebras of $\g$. In particular, $\mathcal{G}_V^+= \g_V^+ \otimes A$, $\mathcal{G}_V^i= \g_V^i \otimes A$, and $\mathcal{G}_V^- = \g_V^-\otimes A $ are subalgebras of $\mathcal{G}$.

\begin{proposition}\label{proposition:twocasesparte1}
Let $V$ be a simple finite weight $\mathcal{G}$-module.
\begin{enumerate}
    \item\label{theoremitem:theoremclassification001} If $\Delta_{\boh} \subset \Delta_V^{f}$, then $V$ is finite-dimensional. In this case, $\Delta = \Delta_V^f$, and $V$ is a simple highest weight module.
    \item If $\Delta_{\boh} \subset \inj V$, then $V$ is bounded. Furthermore, there is a finite subset $\Theta \subset \Supp V$ such that $\Supp V = \Theta + \mathcal{Q}_{\boh}$, and $\dim V^{\mu} = \dim V^{\lambda}$ if there is $\gamma \in \Theta$ such that $\mu, \lambda \in \gamma + \mathcal{Q}_{\boh}$.
\end{enumerate}
\end{proposition}
\begin{proof}

(1): Suppose $\Delta= \Delta^f_V$. Write $\Delta_{\boh}= \{ \alpha_1,\dots,\alpha_t\}$ and $\Delta_{\bon} = \{ \beta_1,\dots, \beta_s \}$. Let $\lambda \in \Supp V$, and define 
$$W_{0}(\lambda) =  \bigwedge \left (\g_{\bon} \otimes A \right )V^{\lambda} =   \bigwedge \left (\g_{\bon} \otimes A/\Ann_A (V) \right )V^{\lambda},$$
$$W_i(\lambda) = U(\g^{\alpha_i} \otimes A)  \dots  U(\g^{\alpha_1} \otimes A) U(\h \otimes A) W_0 (\lambda),$$ where $i=1,\dots,t$. Since $W_i(\lambda)$ is a weight $\h \otimes 1$-module, we can define $S_i(\lambda)$ as its weights. By the PBW theorem and the simplicity of $V$, $W_t(\lambda) = V$ and $S_t(\lambda) = \Supp V$. $W_0(\lambda)$ has finite dimension and $S_0(\lambda)$ is finite, because $ A/\Ann_A (V)$ and, therefore, $\bigwedge \left (\g_{\bon} \otimes A/\Ann_A (V) \right )$ has finite dimension and $V^{\lambda}$ has finite dimension. Therefore, by Proposition \ref{proposition:equivalentnilpotent}, $$S_1 (\lambda) \subset \bigcup_{\gamma \in S_0(\lambda)} \Supp V \cap  \{\gamma + n \alpha_i \mid i=1,\dots,s, \ n\geq 0 \}$$ is a finite set. Suppose, by induction, that $S_i(\lambda)$ is finite. By Proposition \ref{proposition:equivalentnilpotent}, for each $\gamma \in S_i(\lambda)$ the set $\Supp V \cap \{\gamma + n \alpha_i \mid n \geq 0 \}$ is finite. Therefore, $S_{i+1} \subset \bigcup_{\gamma \in S_i(\lambda)} \Supp V \cap \left \{\gamma + n \alpha_i \mid n \geq 0 \right \} $ is a finite union of finite sets, we see $S_{i+1}(\lambda)$ is a finite set as well. Consequently, $\Supp V$ is a finite set. Since $V$ has a finite number of weights and it has finite-dimensional weight spaces, $V$ is a finite-dimensional $\mathcal{G}$-module. 
         
(2): Suppose $\Delta_{\boh} \subset \inj V$. As we did in the first part, we define $$W_{0}(\lambda) =  \bigwedge \left (\g_{\bon} \otimes A \right )V^{\lambda} =   \bigwedge \left (\g_{\bon} \otimes A/\Ann_A (V) \right )V^{\lambda},$$ for some fixed $\lambda \in \Supp V$. By the same argument given in the first part, $W_0(\lambda)$ has finite dimension, thus the set $\Theta$ of its weights is finite. By the PBW Theorem and the simplicity of $V$, we have that $V=\prod_{\alpha \in \Delta_{\boh}}  U(\g^{\alpha} \otimes A) U(\h \otimes A)W_0(\lambda)$, thus $\Supp V = \Theta + \mathcal{Q}_{\boh}$ since $\Delta_{\boh} \subset \inj V$. 
         
Let $\alpha \in \Delta_{\boh}$ and $\gamma \in \Theta$, then $\alpha, -\alpha \in \inj V$ and $V^{\gamma+n\alpha} \neq 0$ for all $n \in \Z$. Hence $V$ has infinite dimension. Since $x_{\alpha} \otimes 1$ acts injectively, the linear map
\begin{align*}
V^{\gamma} &\rightarrow V^{\gamma + \alpha} \\
v&\mapsto (x_{\alpha} \otimes 1) v
\end{align*}
is injective, therefore $\dim V^{\gamma} \leq \dim V^{\gamma+\alpha}$. Likewise, $x_{-\alpha} \otimes 1$ acts injectively on $V^{\gamma+\alpha}$ and thus the linear map \begin{align*}
V^{\gamma + \alpha } &\rightarrow V^{\gamma} \\
v&\mapsto (x_{-\alpha} \otimes 1) v
\end{align*} is injective, which implies that $\dim V^{\gamma  + \alpha} \leq \dim V^{\gamma}$. Repeating this argument we conclude that $\dim V^{\gamma} = \dim V^{\gamma + \beta}$ for all $n \in \Z$ and $\beta \in \mathcal{Q}_{\boh}$. 
		
Since $\Theta$ is finite and $\dim V^{\gamma}$ is finite for all $\gamma \in \Theta$, we conclude $V$ is bounded.
\end{proof}
\begin{corollary}\label{corollary:thereisinjroot}
Assume that $V$ is a simple infinite-dimensional weight $\mathcal{G}$-module. Then $\inj V \cap \Delta_{\boh} \neq \emptyset$.
\end{corollary}

\begin{definition}[Triangular decomposition]\label{def:TD}
A \emph{triangular decomposition} $T$ of $\g$ is a decomposition $\g= \g_T^+ \oplus \g_T^0 \oplus \g_T^-$ such that there exists a linear map $l:\mathcal{Q} \rightarrow \Z$ for which $\g_T^+ = \bigoplus_{l(\alpha)>0} \g^{\alpha}$, $\g_T^0 = \bigoplus_{l(\alpha)=0} \g^{\alpha}$ and $\g_T^- = \bigoplus_{l(\alpha)<0}\g^{\alpha}$. A triangular decomposition $T$ of $\cG$ is a decomposition of the form $\cG = \cG_T^+\oplus \cG_T^0\oplus \cG_T^-$, where $\cG_T^\bullet :=\g_T^\bullet\otimes A$ and $T$ is a triangular decomposition of $\g$. A triangular decomposition is \emph{proper} if $\g_T^0 \neq  \g$. Finally, we set $\Delta_T^{+} := \{\alpha \in \Delta \mid l(\alpha)>0 \}$, $\Delta_T^{^-} := \{\alpha \in \Delta \mid l(\alpha)<0 \}$ and $\Delta_T^{0} := \{\alpha \in \Delta \mid l(\alpha)=0\}$.
\end{definition}

\begin{lemma}\label{lemma:triangdecompindeltai}
    Let $V$ be a simple finite weight $\mathcal{G}$-module.
    \begin{enumerate}
        \item\label{lemmaitem:triangdecompindeltai01} The monoid $\mathcal{Q}^i_V$ generated by all even roots in $\Delta^i_V$ is a group, and for every odd root $\alpha \in \Delta^i_V$ there is $m > 0$ such that $m \alpha \in \mathcal{Q}^i_V$.
        \item\label{lemmaitem:triangdecompindeltai02} There is a triangular decomposition $T$ of $\cG$ such that $\Delta_V^i = \Delta_T^0$, $\Delta_V^+ \subset \Delta_T^+$, and $\Delta_V^- \subset \Delta_T^-$.
    \end{enumerate}
\end{lemma}
\begin{proof}
Since the root system of $\mathcal{G}$ coincides with that of $\g$ and the set $\Delta^i_V$ is completed determined by the action of $\g$ on $V$ (see Proposition \ref{proposition:equivalentnilpotent} and Lemma \ref{lemma:actseverywhere02}), the same proof of \cite[Theorem 3.6]{DMP00} also work in our setting.
 \end{proof}

%-----------------------------------------------------%

Let $T$ be a triangular decomposition of $\cG$. For any weight $\mathcal{G}_T^0$-module $W$, we define the induced module $$M_T(W) = U \left ( \mathcal{G} \right ) \otimes_{U\left (\mathcal{G}^0_T \oplus \mathcal{G}_T^+ \right )} W,$$ where the action of $\mathcal{G}_T^+$ on $W$ is given by $\mathcal{G}_T^+ W=0$.
\begin{proposition}\label{propositionn:generalizedvermatypemodule}
Let $W$ be a weight $\mathcal{G}_T^0 = \g_T^0 \otimes A$-module whose support is included in a single $\mathcal{Q}^T$-coset, where $\mathcal{Q}^T$ is the root lattice of $\g_T^0$.
    \begin{enumerate} 
        \item\label{propitem:generalizedvermatypemodule01} $M_{T}(W)$ has a unique submodule $N_T(W)$ which is maximal among all submodules of $M_{T}(W)$ with trivial intersection with $W$. 
        \item\label{propitem:generalizedvermatypemodule02} $N_{T}(W)$ is maximal if and only if $W$ is simple. In particular, $L_{T}(W) = M_{T}(W) / N_{T}(W)$ is a simple $\mathcal{G}$-module if and only if $W$ is a simple $\mathcal{G}_T^0$-module.  
     \item If $W$ is simple, the space
     $$L_{T}(W)^{\mathcal{G}_T^+} = \left \{v \in L_{T}(W) \mid xv=0 \text{ for all }x \in \mathcal{G}_T^+\right \} $$
    of $\mathcal{G}_T^+$-invariants is equal to $W$.
    \end{enumerate}
\end{proposition}
\begin{proof}
This proof is standard (see \cite[Lemma 2.3, Corollary 2.4]{DMP00}).
\end{proof}
%-----------------------------------------------------%
\begin{theorem}\label{theorem:moduleisinduced}
    Let $V$ be a simple finite weight $\mathcal{G}$-module, then there is a triangular decomposition $\Delta = \Delta_T^- \cup \Delta_T^0 \cup \Delta_T^+$ such that $\Delta_V^i = \Delta_T^0$, $\Delta_V^+ \subset \Delta_T^+$, $\Delta_V^- \subset \Delta_T^-$, the vector space of $\mathcal{G}^+_T$-invariants $V^{\mathcal{G}^+_T}$ is a simple bounded $\mathcal{G}_V^i$-module, and $V \cong L_{T}(V^{\mathcal{G}^+_T})$.
\end{theorem}       
     \begin{proof}
    By Lemma \ref{lemma:triangdecompindeltai}, there is a triangular decomposition $\Delta = \Delta^-_T \cup \Delta^0_T \cup \Delta^+_T$ such that $\Delta_V^i = \Delta^0_T$, $\Delta_V^+ \subset \Delta^+_T$, and $\Delta_V^- \subset \Delta^-_T$, i.e. $\mathcal{G}_T^0 = \mathcal{G}_V^i$, and $\mathcal{G}_V^{\pm} \subset \mathcal{G}_T^{\pm}$. Since $[\mathcal{G}_V^i,\mathcal{G}_T^+] \subset \mathcal{G}_T^+$, $V^{\mathcal{G}_T^+}$ is a submodule of $V$ over $\mathcal{G}_V^i$. First, we will prove that $V^{\mathcal{G}_T^+}$ is a nonzero $\mathcal{G}_V^i$-module, and we will use it to create a $\mathcal{G}_V^i$-module $W$ such that its support is included in a single $\mathcal{Q}^T$-coset. Then, we will prove that $V \cong L_{T}(W)$, $W$ is simple, and $W=V^{\mathcal{G}_T^+}$. Lastly, we will show that $W$ is bounded.
		
	Let $v \in V$ be a nonzero weight vector of weight $\lambda$. Using the fact all roots $ \Delta_V^+$ are locally finite, and the argument given Proposition \ref{proposition:twocasesparte1} \eqref{theoremitem:theoremclassification001}, $U(\mathcal{G}_T^+)v$ is a finite dimensional $\mathcal{G}_T^+$-module. Therefore, $U(\mathcal{G}_T^+)v \cap V^{\mathcal{G}_T^+} \neq  0$, and $V^{\mathcal{G}_T^+}$ is a nonzero $\mathcal{G}_V^i$-module. Let $w \in V^{\mathcal{G}_T^+}$ any weight nonzero vector and let $W=U \left ( \mathcal{G}_V^i \right )w$. Thus, $W$ is a nonzero weight $\mathcal{G}_V^i$-module whose support is included in a single $\mathcal{Q}^T$-coset.

	The linear map 
		\begin{align*}
		    \varphi:M_{T} \left (W \right) &\rightarrow V \\
		    u \otimes_{U\left (\mathcal{G}^i_V \oplus \mathcal{G}_T^+ \right )} v & \mapsto uv
		\end{align*}
	    is a well-defined $\mathcal{G}$-module homomorphism. Furthermore, $\varphi$ is surjective, because its image contains $W$ and $V$ is a simple $\mathcal{G}$-module. Hence $V$ is isomorphic to $M_{T}(W) / \textup{ker} \, \varphi $, and $\textup{ker} \, \varphi$ is a maximal submodule of $M_{T}(W)$. By the PBW theorem, $M_{T}\left (W \right)$ can be identified with $U(\mathcal{G}_T^-) \mathcal{G}_T^- \otimes W \oplus  1 \otimes W$ thus $\varphi(1 \otimes v) =v $ for all $v \in W$, and $\varphi$ restricted to $W$ is a $\mathcal{G}_V^i$-module homomorphism and an injection that preservers weight spaces. Therefore, $W\cap \textup{ker} \, \varphi = 0$, and $\textup{ker} \, \varphi \subset N_T(W)$. By Proposition \ref{propositionn:generalizedvermatypemodule} \eqref{propitem:generalizedvermatypemodule01}, $\textup{ker} \, \varphi =N_{T}\left (W \right)$, because $\textup{ker} \, \varphi$ is maximal. We conclude that $V \cong L_{T}\left (W\right)$. By Proposition \ref{propositionn:generalizedvermatypemodule}, $W$ is simple and it is equal to $V^{\mathcal{G}^+_T}$.

	    Since $\mathfrak{g}^i_V$ is a good Levi subalgebra, we have a isomorphism $$\g^i_V \cong \mathfrak{z} \oplus \mathfrak{l},$$
	    where $\mathfrak{z}$ is a subalgebra of the Cartan subalgebra $\h$ of $\g$, and $\mathfrak{l} \cong \bigoplus_{r=1}^k \mathfrak{l}^r$ is a direct sum of certain simple finite-dimensional Lie superalgebras where at most one $\mathfrak{l}^i$ has nontrivial odd part (see \cite{DMP00, DMP04} for a complete list of good Levi subalgebras). Since $\mathfrak{z}$ is a subalgebra of $\h$, $W$ is simple as a $\g^i_V \otimes A$-module if and only if $W$ is simple as a $\mathfrak{l} \otimes A$-module.
	    Since at most one $\mathfrak{l}^i$ has nontrivial odd part, Theorem \ref{proposition:simpleofsumistensor} implies that $$W|_{\mathfrak{l} \otimes A} \cong  \bigotimes_{r=1}^k W_r ,$$ where $W_r$ is a simple finite weight module over $\mathfrak{l}^r\otimes A$ for each $r=1,\dots,k$. By definition, all even roots on $\Delta^i_V$ are injective. By Proposition \ref{proposition:twocasesparte1}, each $W_r$ is a bounded $\mathfrak{l}^r\otimes A$-module. Therefore, $W$ is bounded $\g_V^i \otimes A$-module, due to the orthogonality between $\mathfrak{l}^r$ and $\mathfrak{l}^s$.
\end{proof}

\begin{definition}[Cuspidal modules]
Let $V$ be a simple weight $\cG$-module. When there is a proper triangular decomposition $T$ of $\cG$ and a simple $\cG_T^0$-module $W$ such that $V \cong L_{T}(W)$, we say $V$ is parabolically induced. If $V$ is not parabolically induced, we call $V$ \emph{cuspidal}. If $\cG$ admits cuspidal modules we say  $\cG$ is a cuspidal superalgebra.
\end{definition}

By Theorem \ref{theorem:moduleisinduced}, every finite weight $\mathcal{G}$-module is either cuspidal, or it is isomorphic to some parabolically induced module of the form $L_T(W)$ with $W$ being a bounded $\mathcal{G}_T^0$-module. However, it is not the case that every parabolically induced module constructed in this way is a finite weight $\mathcal{G}$-module, as can be seen in the following example. 

\begin{example}
Let $\g$ be a basic classical Lie superalgebra and $A = \kc [t]$. Any usual triangular decomposition $\Delta = \Delta^+\cup \Delta^-$ defines a triangular decomposition $T$ of $\Delta$ in the sense of Definition~\ref{def:TD}. In particular, we have $\mathcal{G}_T^0=  \h \otimes \kc[t]$. Fix a nonzero element $h\in \h$, and let $\Lambda: \h \otimes \kc[t]\rightarrow \kc$ be a linear functional satisfying $\Lambda(h \otimes t^k )= \frac{1}{k+1}$ for all $k \geq 0$. The 1-dimensional $\mathcal{G}^0_T$-module $\kc v_{\Lambda}$ defined by $xv_{\Lambda} = \Lambda(x)v_{\Lambda}$ is a simple $\mathcal{G}_T^0$-module with finite-dimensional weight-spaces, but the simple highest weight $\mathcal{G}$-module $L_T(\kc v_{\Lambda})$ does not have finite-dimensional weight spaces by \cite[Theorem 4.16]{Sav14}. 
\end{example}

\begin{remark}\label{remark:giscuspidalevi}
	If $\Delta_V^i \neq \emptyset$, then the $\g_V^i \otimes A$-module $V^{\mathcal{G}_T^+}$ is cuspidal because all roots on $\Delta_{V}^i$ acts injectively on $V$, consequently on $V^{\mathcal{G}_T^+}$ as well. By Proposition \ref{theorem:moduleisinduced}, we have that $V^{\mathcal{G}_T^+}$ is bounded, and hence it is bounded when restricted to $\g_V^i$. Thus, it has finite length over $\g^i_V$ by \cite[Lemma 6.2]{DMP00}. Now, if $M$ is a simple $\g_V^i$-submodule of $V^{\mathcal{G}_T^+}$, then $M$ is cuspidal.  In particular, $\g_V^i \cong \mathfrak{z} \oplus \mathfrak{l}$ is a cuspidal Levi superalgebra, with $\mathfrak{z} \subset \h$ and $\mathfrak{l}$ isomorphic to $\mathfrak{osp}(1 | 2)$, $\mathfrak{osp}(1|2) \oplus \mathfrak{sl}_2$, $\mathfrak{osp}(n|2n)$ with $2 < n \leq 6$, $D(2,1;a)$, or a reductive Lie algebra with irreducible components of type $A$ and $C$ \cite{RF09}. Note that at most one of the simple components of $\mathfrak{l}$ is a Lie superalgebra.
	With Theorem \ref{proposition:simpleofsumistensor} and \cite{BLL15,Lau18} in mind, we conclude that it remains to classify bounded cuspidal modules over the map superalgebra associated to these basic Lie superalgebras, since every simple module over $(\g_V^i \oplus \mathfrak{l}) \otimes A \cong (\mathfrak{z} \otimes A) \oplus (\mathfrak{l} \otimes A) $ is a tensor product of simple representations over the map superalgebras associated to the simple components of $\mathfrak{l}$.
\end{remark}

\begin{remark}
	The notion of cuspidal Levi superalgebras considered in \cite{DMP00,DMP04} is more general than the one we are using in the current paper, since their notion of parabolic induced modules is more general  (in \cite{DMP00,DMP04}, parabolic induction is performed over \emph{generalized} weight modules). The notion of cuspidal modules considered here is aligned with that of \cite{RF09}, and therefore the list of cuspidal Levi superalgebras presented in \cite{RF09} is precisely the one we must consider. However, we point out that the argument given in Remark \ref{remark:giscuspidalevi} and the results proven in Section \ref{sectioon:cuspidaluniformboundmodules} are also valid for the list of cuspidal Levi superalgebras from \cite{DMP00,DMP04}, since for all such superalgebras  at most one of its simple components is a Lie superalgebra with semisimple even part. 
\end{remark}

    \section{Cuspidal Bounded Modules}
    \label{sectioon:cuspidaluniformboundmodules}

In this section we complete the classification given in Theorem \ref{theorem:moduleisinduced}. 
By Remark \ref{remark:giscuspidalevi}, we just need to classify all cuspidal bounded simple modules over $\g \otimes A$, where $\g$ is isomorphic to $\mathfrak{osp}(1 | 2)$, $\mathfrak{osp}(1|2) \oplus \mathfrak{sl}_2$, $\mathfrak{osp}(n|2n)$ with $2 < n \leq 6$, $D(2,1;a)$, or to a reductive Lie algebra with irreducible components of type $A$ and $C$. Since the case where $\g$ is isomorphic to a Lie algebra is known \cite{BLL15}, it only remains to classify all cuspidal bounded simple modules over $\g \otimes A$, for $\g$ isomorphic to $\mathfrak{osp}(1|2)$, $\mathfrak{osp}(n|2n)$ with $2 < n \leq 6$, or $D(2,1;a)$.

    \begin{proposition}\label{proposition:kernelisradicaluniform}
        Suppose that $\g$ is a basic classical Lie superalgebra such that $\g_{\boh}$ is semisimple. If $V$ is a simple bounded weight $\g \otimes A$-module, then the ideal $\Ann_A(V)$ is radical.
    \end{proposition}
    \begin{proof}
        Set $I:=\Ann_A(V)$. Then $A/I$ is a finite-dimensional commutative algebra and $V$ is a faithful representation of $\g \otimes A/I$. Since $A/I$ is artinian, its nilradical $J=\sqrt{I}/I$ is nilpotent. If $J=0$, then $I=\sqrt{I}$ is radical and the proof is complete. Seeking a contradiction, let's assume that $J \neq 0$. Then there is $m>0$ such that $J^{m-1} \neq 0$, but $J^m=0$. Let $m'$ be the smallest positive integer greater or equal to $m/2$, and set $N=J^{m'}$. Thus $N\neq 0$, and $N^2=0$.
    
        \underline{\textbf{Claim 1:}} \textit{$\g^{\alpha} \otimes N$ acts nilpotently on $V$ for every $\alpha \in \Delta$.}
        
        If $\alpha \in \Delta_{\bon}$, $x \in \g^{\alpha}$, and $a \in N$, then
        \[
            0 = ([x,x] \otimes a^2 )u = [x\otimes a,x\otimes a]u = 2(x \otimes a)^2u,
        \]
        for all $u \in V$, because $a^2 \in N^2 = 0$. Therefore, $(x\otimes a)^2u =0$, and $\g^{\alpha} \otimes N$ acts nilpotently on $V$ for every odd root $\alpha \in \Delta$. If $\alpha$ is even, the proof that $\g^{\alpha} \otimes N$ acts nilpotently on $V$ follows from Step 2 of \cite[Proposition 4.4]{BLL15}.

	\underline{\textbf{Claim 2:}} \textit{There is a nonzero weight vector $w \in V$ such that $(\g \otimes N)w=0$.}
	
	Let $\lambda \in \Supp V$. Note that $\h \otimes N$ is an abelian Lie algebra with finite dimension that commutes with $\h \otimes 1$, thus $\left (\h \otimes N  \right )V^{\lambda} \subset V^{\lambda}$, and, by Lie's Theorem, there is a nonzero vector $v_0 \in V^{\lambda}$ such that $ (\h \otimes N )v_0\subset \kc v_0$. 
	
    Let $\Delta=\Delta^+ \cup \Delta^-$ be a choice positive roots of $\Delta$ and set $\mathfrak{n}^{\pm} := \bigoplus_{\alpha \in \Delta^{\pm}} \g^{\alpha}$. Recall that $N$ is an ideal of $A/I$, $N \neq 0$, and $N^2=0$, thus $[\g \otimes N, \g \otimes N] \subset [\g,\g] \otimes N^2 = 0$. By Claim~1, every element of the finite-dimensional abelian Lie algebra $\mathfrak{n}^+ \otimes N$ acts nilpotently on $V$. Thus $U(\mathfrak{n}^+ \otimes N)v_0$ is a finite-dimensional $(\mathfrak{n}^+ \otimes N)$-module. By Engel's theorem, there is a nonzero vector $u \in U(\mathfrak{n}^+\otimes N)v_0$ such that $(\mathfrak{n}^+ \otimes N)u=0$. Since $U(\mathfrak{n}_+ \otimes N)v_0 \subset V$ is a weight module, we can assume that $u$ is a weight vector. The finite-dimensional abelian Lie algebra $\mathfrak{n}^- \otimes N$ acts nilpotently on $u$, and a similar argument shows that there is a nonzero weight vector $w \in U(\mathfrak{n}^- \otimes N)u$ such that $(\mathfrak{n}^- \otimes N)w=0$. Since $U(\mathfrak{n}^+ \otimes N)$ commutes with $U(\mathfrak{n}^- \otimes N)$ (due to $N^2 =0$), we have $(\mathfrak{n}^+ \otimes N)w=0$. Since $(\h \otimes N)v_0 \subset \kc v_0$, $w \in U( (\mathfrak{n}^- \oplus \mathfrak{n}^+) \otimes N)v_0$, and $\h \otimes N$ commutes with $U( (\mathfrak{n}^- \oplus \mathfrak{n}^+) \otimes N)$, we conclude that $w$ is an eigenvector for every element of $\h \otimes N$. 
	
	It remains to show that $(\h \otimes N)w=0$. The proof that $(h_{\alpha} \otimes t)w =0$ for all $t \in N$ and $\alpha \in \Delta_{\boh}$ is the same as that of Step 5 of \cite[Proposition 4.4]{BLL15}. Since we are assuming $\g_{\boh}$ is a semisimple Lie algebra and $\h$ is a Cartan subalgebra of $\g_{\boh}$ (by definition), the set $\{ h_{\alpha} \mid \alpha \in \Delta_{\boh} \}$ generates $\h$ as a vector space. Hence $(\h \otimes N)w=0$ as we wanted.

	\underline{\textbf{Claim 3:}} \textit{The ideal $I$ is radical.}
	
	 Let $W = \{ w \in V \mid (\g \otimes N)w = 0 \}$. Since $N$ is an ideal, for all $(x \otimes a) \in \g \otimes A/I$ and $y \otimes b \in \g \otimes N$ we get that
    	\[
    (y \otimes b) (x \otimes a) w = (-1)^{|x| |y|}(x \otimes a)(y \otimes b) w +([x,y] \otimes ba )w =0.
        \]
	Together with Claim 2, we obtain that $W$ is a nonzero $\g \otimes A/I$-submodule of $V$. Since $V$ is simple, $W=V$, and hence $0 \neq N \subset \Ann_{A/I}(V)$, which contradicts the fact that $V$ is a faithful representation of $\g \otimes A/I$ (since $(\g \otimes N )V=0$). Hence $J=\sqrt{I}/I=0$, which implies that $I=\Ann_{A}(V)$ is a radical ideal of $A$.
    \end{proof}

    \begin{definition}[Evaluation modules]
    Suppose $\mathsf{m}_1,\dots,\mathsf{m}_r \in \Specm A$ are pairwise distinct. The associated \emph{evaluation map} $\textup{ev}_{\mathsf{m}_1,\dots,\mathsf{m}_r}$ is the composition
    $$\textup{ev}_{\mathsf{m}_1,\dots,\mathsf{m}_r} : \mathcal{G} \rightarrow \mathcal{G} / \left (\g \otimes \prod_{i=1}^r \mathsf{m}_i \right )\cong  \bigoplus_{i=1}^r \left (\g \otimes A/\mathsf{m}_i \right ) \cong \g^r.$$
    If $V_i$ is a $\g$-module with respective representations $\rho_i: \g \rightarrow \textup{End} \, V_i$ for each $i=1,\dots,r$, then the composition $$\mathcal{G} \xrightarrow{\textup{ev}_{\mathsf{m}_1,\dots,\mathsf{m}_r}} \g^r \xrightarrow{\bigotimes_{i=1}^r \rho_i} \textup{End}\left (\bigotimes_{i=1}^r V_i \right) $$
    is a representation of $\mathcal{G}$ called an \emph{evaluation representation}. The corresponding module is called an \emph{evaluation module}, and is denoted by $$\bigotimes_{i=1}^r V_i^{\mathsf{m}_i}.$$ 
    \end{definition}
    
    For each simple finite weight $\mathcal{G}$-module (or $\g$-module) $V$, we denote $R_V= C_V \cap \Delta$. Notice that if $V$ is cuspidal, then $R_V$ is equal to $\Delta$.

\begin{lemma}\label{lemma:maximunoneuniformbounded}
Let $r\geq 2$, and $V_1,\ldots, V_r$ be simple bounded weight $\g$-modules such that $\Delta = \cup_{i=1}^r R_{V_i}$. Then the $\g$-module $V:=\bigotimes_{i=1}^r V_i$ is bounded if and only if $\dim V_i = \infty$ for at most one $i=1,\ldots, r$.
\end{lemma}
\begin{proof}
It is clear that if $\dim V_i = \infty$ for at most one $i=1,\ldots, r$, then $V$ is bounded.

Assume without loss of generality that $\dim V_1 =  \dim V_2 = \infty$. Let's prove that $V$ is not bounded. Notice that the fact that $R_{V_1}$ is closed implies that $\g^1 := \bigoplus_{\alpha\in R_{V_1}} h_\alpha \oplus \bigoplus_{\alpha\in R_{V_1}} \g^\alpha$ is a subalgebra of $\g$.

If $(R_{V_1} + \cup_{i=2}^r R_{V_i})\cap \Delta = \emptyset$ or $(R_{V_1} + \cup_{i=2}^r R_{V_i})\cap \Delta \subseteq R_{V_1}$, then $\bigoplus_{\alpha\in R_{V_1}} h_\alpha \oplus \bigoplus_{\alpha\in R_{V_1}} \g^\alpha$ is an ideal of $\g$, which gives $\g = \g^1$, since $\g$ is simple. Thus, $\Delta = R_{V_1}$ and $R_{V_2}\subseteq R_{V_1}$. For $\alpha\in R_{V_2}$ (which is nonempty, as $\dim V_2=\infty$), pick $m\in \Z_+$ so that $\widetilde{\alpha} = m\alpha\in (\Z_+ \inj V_1)\cap (\Z_+ \inj V_2)$. Now, for any $r$-tuple of weights $(\lambda_1,\ldots, \lambda_r)\in \times_{i=1}^r \Supp V_i$, we have
    \[
\lim_{n\to \infty} (\dim V^{\lambda_1+\cdots + \lambda_r + n\widetilde{\alpha}}) = \infty
    \]
as, for each $n\geq 0$, the following inclusion holds
    \[
\bigoplus_{l+k=n}^n V_1^{\lambda_1 + l\widetilde{\alpha}}\otimes V_2^{\lambda_2 + k\widetilde{\alpha}}\otimes V_3^{\lambda_3}\otimes \cdots \otimes V_r^{\lambda_r}\subseteq V^{\lambda_1+\cdots + \lambda_r + n\widetilde{\alpha}}.
    \]

Assume now that there are $\alpha\in R_{V_1}$ and $\beta\in R_{V_j}$ for $j\geq 2$ such that $\alpha+\beta\in \cup_{i=2}^r R_{V_i}$. Suppose without loss of generality that $j=2$ and fix any $r$-tuple of weights $(\lambda_1,\ldots, \lambda_r)\in \times_{i=1}^r \Supp V_i$. We claim that there is $m \in \Z_+$ such that $\widetilde{\alpha} := m\alpha$ and $\widetilde{\beta} := m\beta$ satisfy the following condition
    \[
\lim_{n\to \infty} \dim \left (V^{\lambda_1+\cdots + \lambda_r + n(\widetilde{\alpha}+\widetilde{\beta} )} \right ) = \infty
    \]
Indeed, for all $n\geq 0$, one of the following hold:

(1): $\alpha+\beta\in R_{V_2}$. In this case, we take $m\in \Z_+$ for which $\widetilde{\alpha} = m\alpha\in \Z_+\inj V_1$,  $\widetilde{\beta} = m\beta \in \Z_+ \inj V_2$ and $m(\alpha+\beta) = \widetilde{\alpha}+\widetilde{\beta} \in \Z_+ \inj V_2$. Then
    \[
\bigoplus_{l=0}^n V_1^{\lambda_1 + l\widetilde{\alpha}}\otimes V_2^{\lambda_2 + l\widetilde{\beta} +  (n-l)(\widetilde{\alpha}+\widetilde{\beta})}\otimes V_3^{\lambda_3}\otimes \cdots \otimes V_r^{\lambda_r}\subseteq V^{\lambda_1+\cdots + \lambda_r + n(\widetilde{\alpha}+\widetilde{\beta})}
    \]
    
(2): $\alpha+\beta\in R_{V_k}$ for $k=3,\ldots, r$, and (for convenience of notation, assume that $k=3$). Similarly to the previous case, take $m\in \Z_+$ such that $\widetilde{\alpha} = m\alpha\in \Z_+\inj V_1$,  $\widetilde{\beta} = m\beta \in \Z_+ \inj V_2$ and $m(\alpha+\beta) = \widetilde{\alpha}+\widetilde{\beta} \in \Z_+ \inj V_3$. Then
    \[
\bigoplus_{l=0}^n V_1^{\lambda_1 + l\widetilde{\alpha}}\otimes V_2^{\lambda_2 + l\widetilde{\beta}}\otimes V_3^{\lambda_3 +  (n-l)(\widetilde{\alpha}+\widetilde{\beta})}\otimes V_4^{\lambda_4}\otimes \cdots\otimes V_r^{\lambda_r}\subseteq V^{\lambda_1+\cdots + \lambda_r + n(\widetilde{\alpha}+\widetilde{\beta})}
    \]
    
In any case, we conclude that $\lim_{n\to \infty} (\dim V^{\lambda_1+\cdots + \lambda_r + n(\widetilde{\alpha}+\widetilde{\beta})}) = \infty$.
\end{proof}

     \begin{theorem}\label{theorem:uniformlyboundisevaluation}
        Suppose $\g_{\boh}$ is semisimple (in particular, this is the case when $\cG$ admits cuspidal modules). If $V$ is a simple cuspidal bounded $\mathcal{G}$-module, then $V$ is isomorphic to an evaluation module.
    \end{theorem}
    \begin{proof}
    By Proposition \ref{proposition:kernelisradicaluniform}, $I=\Ann_A (V)$ is radical, so there are distinct maximal ideals $\mathsf{m}_1,\dots,\mathsf{m}_r$ of $A$ such that $I=\bigcap_{i=1}^r \mathsf{m}_i$. By the Chinese Remainder Theorem, $$\mathcal{G} / \left ( \g \otimes \bigcap_{i=1}^r \mathsf{m}_i \right )  \cong  \g \otimes \left (  A/\mathsf{m}_1 \oplus \dots \oplus A/\mathsf{m}_1 \right ) \cong \bigoplus_{i=1}^r \left (\g \otimes A /\mathsf{m}_i \right ) \cong \g^{\oplus r},$$
    where the last isomorphism follows from the fact that $A$ is a finitely generated algebra over the algebraically closed field $\kc$.
    In particular, $V$ is a simple module over $\g^{\oplus r}$.
    By Theorem~\ref{proposition:simpleofsumistensor},  $V \cong V_1 \hat{\otimes} \dots \hat{\otimes} V_r$, where $V_1,\dots,V_r$ are simple finite weight $\g$-modules.
    
    Now, notice that the $\g$-module $V_1 \hat{\otimes} \dots \hat{\otimes} V_r$ is bounded if and only if the $\g$-module $\widetilde{V}:=V_1 \otimes \dots \otimes V_r$ is bounded. Moreover, since $V$ is cuspidal, we have that $\Delta = R_V$. Since either $\widetilde{V}\cong V\oplus V$ or $\widetilde{V}\cong V$ (see Proposition~\ref{proposition:chengresultfortensorproduct}) we conclude that $R_V = \bigcup_{i=1}^r R_{V_i}$. Indeed, this follows from Lemma~\ref{lemma:dasalphastring} along with the fact that $\Supp \widetilde{V} = \Supp V \times \Supp V$, and $\widetilde{V}
   ^\lambda = \bigotimes_{\lambda_1+\cdots + \lambda_r = \lambda} V_i^{\lambda_i}$.  Thus, by Lemma \ref{lemma:maximunoneuniformbounded}, at most one of the $V_i$ has infinite dimension. Without loss of generality, we assume that $ \dim V_1 = \infty$, and $\dim V_i < \infty$ for each $i=2,\dots,r$. In particular, for all $i>1$, the $\g$-modules $V_i$ are highest weight modules, which implies that $\textup{End}_{\g} \, V_i \cong \kc$. Now Theorem \ref{proposition:simpleofsumistensor} gives us that $$V \cong V_1 \hat{\otimes} \dots \hat{\otimes} V_r = V_1 \otimes \dots \otimes V_r.$$
    
    In terms of representations, we obtained that the representation $\rho: \mathcal{G} \rightarrow \textup{End}(V)$ associated to the $\mathcal{G}$-module $V$ factors through the following composition 
    \begin{align*}
        \mathcal{G} \rightarrow \g \otimes \left (A/ \Ann_A (V) \right ) \xrightarrow{\textup{ev}_{\mathsf{m}_1,\dots,\mathsf{m}_r}} \g^{\oplus r} \xrightarrow{\otimes_{i=1}^r \rho_i }  \textup{End} \left ( \bigotimes_{i=1}^r V_i \right ),
    \end{align*}
    where $\rho_i: \g \rightarrow \textup{End} (V_i) $ is the representation associated to the $\g$-module $V_i$ for each $i=1,\dots,r$. But this composition is precisely the representation associated to the evaluation module $\bigotimes_{i=1}^r V_I^{\mathsf{m}_i}$. The proof is completed. 
    \end{proof}
 
\begin{proposition}\label{proposition:uniformmyifandonlyatmostone}
Let $V_1,\dots,V_r$ be simple finite weight $\g$-modules and $\mathsf{m}_1,\dots,\mathsf{m}_r$ pairwise distinct maximal ideals of $A$. Then the evaluation $\cG$-module $V := \bigotimes_{i=1}^r  V_i^{\mathsf{m}_i}$ is a cuspidal bounded module only if $\dim V_i=\infty$ for precisely one $V_i$, in which case $V_i$ is a cuspidal bounded $\g$-module. In particular, if $V$ is a cuspidal bounded $\mathcal{G}$-module, then it is simple.
\end{proposition}
\begin{proof}
If $V_1,\dots,V_r$ are bounded and no more than one of them is infinite-dimensional, then it is clear that $V$ is bounded. On the other hand, if $V $ is bounded and $N>0$ is such that $ \dim V^{\lambda} \leq N$ for each $\lambda \in \Supp V$, then the dimension of the weight spaces of each $V_i$ has to be less or equal to 
$N$ as well. Thus each $V_i$ is bounded, and, by Lemma \ref{lemma:maximunoneuniformbounded}, no more than one $V_i$ can be infinite-dimensional. This proves the first statement. 
    
The second statement follows from Theorem \ref{proposition:simpleofsumistensor} along with the fact that $\textup{End}_\g (V_i) \cong \kc$ for all, except at most one $i=1,\ldots, r$.
\end{proof}

\section{Affine Kac-Moody Lie superalgebras}
\label{section:applicationtoaffineliesuperalgebras}

\subsection{Finite weight modules over affine Lie superalgebras} The affine Lie superalgebra $\cA(\g)$ associated to a basic classical Lie superalgebra $\g$ is the universal central extension of the loop superalgebra $L(\g):=\g\otimes \C[t, t^{-1}]$. Concretely, if $(\cdot ,\cdot )$ denotes an invariant supersymmetric non-degenerate even bilinear form on $\g$, then $\cA(\g)$ is the vector space $L(\g)\oplus \C c$ for an even element $c$, and the bracket is given by 
    \[
[c, \cA(\g)]=0,\quad [x\otimes t^n, y\otimes t^m] = [x,y]\otimes t^{m+n} + n(x,y)\delta_{n,-m}c,
    \]
for any $x\otimes t^n,\ y\otimes t^m\in L(\g)$. Notice that we have a short exact sequence of Lie superalgebras
    \[
0\to \C c\to \cA(\g)\to L(\g)\to 0.
    \]

Recall that $\cA(\h) := \h \oplus \C c$ is a Cartan subalgebra of $\cA(\g)$. It is clear that any $\cA(\g)$-module $V$ on which $c$ acts trivially is naturally a module over the Lie superalgebra $\cA(\g)/\C c\cong L(\g)$. On the other direction, any $L(\g)$-module $V$ can be made into a $\cA(\g)$-module with trivial action of $c$. Let $V_{c=0}$ denote such $\cA(\g)$-module.

The main result of this section is the following:

\begin{theorem}
The assignment that maps an $L(\g)$-module $V$ to the $\cA(\g)$-module $V_{c=0}$ provides a bijection between the set of simple finite weight $L(\g)$-modules and the set of simple finite weight $\mathcal{L}(\g)$-modules. Moreover, this bijection restricts to a correspondence between the respective sets simple cuspidal modules.
\end{theorem}
\begin{proof}
The first statement will follow if we show that $c$ acts trivially on any simple finite weight $\cA(\g)$. To prove this, we notice that $c$ must act as a scalar $k$ on any simple $\cA(\g)$-module by Lemma~\ref{lemma:lemmadeschur}. We claim that $k$ is zero. Indeed, let $h\in \h$ be a nonzero element for which $(h,h)\ne 0$. Let $V^\lambda$ any weight space. Since $\dim V^\lambda <\infty$, we can take the trace of $[h\otimes t, h\otimes t^{-1}]\in \textup{End}_{\kc} (V^\lambda)$, which is zero. On the other hand, the equality
    \[
[h\otimes t, h\otimes t^{-1}] = (h,h)c
    \]
shows that this trace is precisely $\left(\dim V^\lambda\right) (h,h)k$. The first statement follows.

The fact that $c$ acts trivially on any simple finite weight $\cA(\g)$-module implies the second statement.
\end{proof}

\subsection{Bounded modules over affine Kac-Moody Lie superalgebras}
%It was proved in \cite{RF09} that any simple finite weight module of $\widehat{\g}$ of nonzero level is parabolically induced from a cuspidal simple module over some cuspidal subalgebra of $\g$. {\color{blue} In some cases 
%this gives a reduction to the even part of $\widehat{\g}$, e.g. in the case of $\widehat{\mathfrak{sl}(n|m)}$. This phenomenon was discussed for general superalgebras in \cite{CM21} with finite-dimensional odd part, in which case simple modules are obtained as irreducible quotients of the well known Kac modules. 
In this subsection we assume that $\g$ is a basic classical Lie superalgebra of type I in order to define analogs of  Kac $\widehat{\g}$-modules in the affine setting. 

Consider the $\Z$-grading $\g = \g_{-1}\oplus \g_0\oplus \g_1$, and recall that $\g_0$ is a reductive Lie algebra with 1-dimensional center, that $\g_{\pm 1}$ are simple $\g_0$-modules, and that $[\g_{\pm 1},\g_{\pm 1}]=\{0\}$. Following \cite{CM21}, for a simple weight $\widehat{\g}_0$-module $S$ we define the \emph{Kac module} associated to $S$ to be the induced $\widehat{\g}$-module
    \[
K(S) := U(\widehat{\g})\otimes_{U(\widehat{\g}_0\oplus \widehat{\g}_1)}S,
    \]
where we are assuming that $\widehat{\g}_1 S=0$. It is easy to prove that $K(S)$ admits a unique simple quotient which we denote by $V(S)$. This defines the \emph{Kac induction functor} from the category of 
$\widehat{\g}_0$-modules to the category of $\widehat{\g}$-modules, and we have the following
 main result of this section.

\begin{theorem}\label{thm:Kac.modules.reduction} Let $\g$ be a basic classical Lie superalgebras of type I. Then the  functor $K$ defines 
  a bijection between the sets of isomorphism classes of simple bounded Harish-Chandra $\widehat{\g}$-modules and 
 simple bounded Harish-Chandra $\widehat{\g}_0$-modules. 
\end{theorem}

Theorem \ref{thm:Kac.modules.reduction} reduces  the classification of simple bounded $\widehat{\g}$-modules to the even part of $\widehat{\g}$. It follows immediately from the following lemma.

\begin{lemma}
Let $M$ be a simple bounded weight $\widehat{\g}$-module of level zero. Then the set of $\widehat{\g}_1$-invariants $M^{\widehat{\g}_1} = \{m\in M\mid \widehat{\g}_1 m =0\}$ is a nonzero simple bounded weight $\widehat{\g}_0$-module and $M\cong V(M^{\widehat{\g}_1})$.
\end{lemma}

\begin{proof}
It is enough to show that the vector space $M^{\widehat{\g}_1}$ is nonzero as the rest of the proof will follow by standard arguments.

Assume that all dimensions are bounded by $d\geq 1$, and fix some odd root $\alpha$ of $\g_1$. Without loss of generality assume that we have a weight vector $v\neq 0$ such that $x_{\alpha+s\delta}v=0$ for all $s=0, \ldots 2d-1$. Suppose $w=x_{\alpha+{2d}\delta}v\neq 0$, and consider the vectors
    \[
v, x_{\delta}x_{-\delta}v, \ldots, x_{d\delta}x_{-d\delta} v.
    \]
Since these vectors are linearly dependent, there exists $l\leq d$ such that $x_{l\delta}x_{-l\delta} v=\sum_{i=1}^{l-1} c_i x_{i\delta}x_{-i\delta} v +c_0 v$. Apply $x_{\alpha+(2d-l)\delta}$ on this equation. On the right hand side we obtain $0$, and on the left hand side:
    \[
x_{\alpha+(2d-l)\delta} x_{l\delta}x_{-l\delta} v=x_{\alpha+2d\delta}x_{-l\delta} v=x_{-l\delta}x_{\alpha+2d\delta}v.
    \]
Thus $x_{-l\delta}x_{\alpha+2d\delta}v=x_{-l\delta}w=0$. From this we conclude that $x_{\alpha+s\delta}w=0$ 
for all $s\leq 2d$.

Now, choose the smallest $N\in \Z_{>2d}$ such that $x_{\alpha+N\delta}w\neq 0$, and consider the linearly dependent elements of $M$
    \[
x_{ld\delta}x_{\alpha+N\delta}w, \, x_{l(d-1)\delta}x_{\alpha+(N+l)\delta}w, \ldots, x_{l\delta}x_{\alpha+(N+l(d-1))\delta}w, \, x_{\alpha+(N+ld)\delta} w,
    \]
where $1\leq l\leq d$ was fixed above.    Note that $x_{\alpha+(N+lm)\delta}w\neq 0$ for all $0\leq m\leq d$, since otherwise $x_{-l\delta}^m x_{\alpha+(N+lm)\delta}w$ equals  $x_{\alpha+N\delta}w$ up to a nonzero scalar and hence $x_{\alpha+N\delta}w=0$, which is a contradiction.

Assume that for $0\leq k\leq d$ holds
    \[
x_{l(d-k)\delta}x_{\alpha+(N+lk)\delta}w=\sum_{i=0}^{k-1} c_i x_{l(d-i)\delta} x_{\alpha+(N+il)\delta} w
    \]
and apply $x_{-l\delta}^k$  on this equation. 

Since  $c$ acts trivially on $M$, $x_{-l\delta}w=0$ and $$x_{l(d-i)\delta}x_{\alpha+(N+(i-k)l)\delta} w=0$$ for all $i<k$, we have that
$x_{-l\delta}^k x_{l(d-k)\delta}x_{\alpha+(N+lk)\delta}w=0$. But the left hand side equals 
$x_{l(d-k)\delta}x_{\alpha+N\delta}w$ up to a nonzero scalar.
 Since $v'=x_{\alpha+N\delta}w\neq 0$ by assumption, we obtain $x_{l(d-k)}v'=0$ for some $0\leq k< d$. Taking into account that $x_{-l\delta}v'=x_{-l\delta}x_{\alpha+N\delta}w=0$
we conclude that $x_{\alpha+n\delta}v'=0$ for all $n\in \Z$. 

Repeating the argument above (but now starting with $v'$ instead of $v$) for an odd root $\gamma\neq \alpha$ of $\g_1$ will yield a nonzero vector $v''$ such that $x_{\alpha+n\delta}v''=x_{\gamma+n\delta}v''=0$ for all $n\in \Z$. 
Since the set of roots of $\g_1$ is finite, after finitely many steps we find a nonzero element of $M^{\widehat{\g}_1}$, as we wanted.

Finally, to conclude the proof we observe that if $x_{\alpha+n\delta}v = 0$ for all $n\in \Z_{\geq 0}$, that is, if there is no such vector $w$, then the same argument as above, with positive multiples of $\delta$ replaced by negative multiples of $\delta$, will prove the statement. The proof is now complete.
\end{proof}

\begin{remark}
Notice that Theorem~\ref{thm:Kac.modules.reduction} reduces the classification of all simple bounded weight $\widehat{\g}$-modules to the classification of the same class of modules over $\widehat{\g}_0$. Therefore, the combination of our results with those in the unpublished work  \cite{DG10} provides a classification of all simple bounded weight $\widehat{\g}$-modules.
\end{remark}

\section*{Acknowledgments}	%\addcontentsline{toc}{chapter}{References}
L. C. was supported by the Fapemig (APQ-02768-21). V.F. was supported in part by NSF China - 12350710178.   H.\, R.\, was supported in part by the Coordenação de Aperfeiçoamento de Pessoal de Nível Superior - Brasil (Capes) - Finance Code 001 and by the Fapesp (2020/13811-0). This work was supported in part by the CNPq (402449/2021-5).

%\bibliographystyle{myalpha}
%\bibliography{mybib}
%\addcontentsline{toc}{chapter}{References}

\end{document}